\documentclass[11pt,letterpaper]{article}
\usepackage{amsmath,amssymb,amsthm}
\usepackage{MnSymbol} 
\usepackage{graphicx,color}
\usepackage[hyphens]{url}
\usepackage{dsfont}
\usepackage{booktabs}
\usepackage[normalem]{ulem}
\usepackage{mathtools}
\usepackage[hidelinks,colorlinks=true,linkcolor=blue,citecolor=blue]{hyperref}
\usepackage[nameinlink,capitalize]{cleveref}
\usepackage{multirow}
\usepackage{algorithm}
\usepackage[noend]{algpseudocode}
\usepackage{tikz} %

\usepackage{enumitem}

\usepackage{soul} %
\usepackage[square,sort,numbers]{natbib}
\usepackage[sort,nocompress,noadjust]{cite}

\usepackage{subcaption}

\usepackage{tabularx}
\newcolumntype{L}{>{\raggedright\arraybackslash}X}

\numberwithin{equation}{section}
\newtheorem{theorem}{Theorem}[section]

\newtheorem{lemma}[theorem]{Lemma}
\newtheorem{remark}[theorem]{Remark}

\newtheorem{example}[theorem]{Example}

\newtheorem{defin}[theorem]{Definition}

\newcommand{\R}{\mathbb{R}}

\newcommand{\N}{\mathbb{N}}

\newcommand*{\eps}{\varepsilon}

\newcommand*{\lam}{\lambda}
\newcommand*{\sig}{\sigma}

\renewcommand{\leq}{\leqslant}
\renewcommand{\geq}{\geqslant}
\DeclareMathOperator{\half}{\frac{1}{2}}

\DeclareMathOperator{\bin}{\nu}

\usepackage[margin=1in]{geometry}

\makeatletter
\def\blfootnote{\gdef\@thefnmark{}\@footnotetext}
\makeatother

\begin{document}

\title{Acceleration by Stepsize Hedging II: \\ Silver Stepsize Schedule for Smooth Convex Optimization}

	 \author{
	 	Jason M. Altschuler
	 	\\	UPenn \\	\texttt{alts@upenn.edu}
	 	\and
	 	Pablo A. Parrilo \\
	 	LIDS - MIT \\	\texttt{parrilo@mit.edu}
	 }
	 \date{}
	\maketitle

\begin{abstract}	
    We provide a concise, self-contained proof that the Silver Stepsize Schedule proposed in Part~I~\citep{alt23hedging1} directly applies to smooth (non-strongly) convex optimization. Specifically, we show that with these stepsizes, gradient descent computes an $\eps$-minimizer in $O(\eps^{-\log_{\rho} 2}) = O(\eps^{-0.7864})$ iterations, where $\rho = 1+\sqrt{2}$ is the silver ratio. This is intermediate between the textbook unaccelerated rate $O(\eps^{-1})$ and the accelerated rate $O(\eps^{-1/2})$ due to Nesterov in 1983. The Silver Stepsize Schedule 
    is a simple explicit fractal: the $i$-th stepsize is $1 + \rho^{\nu(i)-1}$ where $\nu(i)$ is the $2$-adic valuation of~$i$. The design and analysis are conceptually identical to the strongly convex setting in~\citep{alt23hedging1}, but simplify remarkably in this specific setting.
\end{abstract}

\section{Introduction}\label{sec:intro} 

We revisit the classical problem of smooth convex optimization: solve $\min_{x \in \R^d} f(x)$ where $f$ is convex and $M$-smooth (i.e., its gradient is $M$-Lipschitz). A celebrated result is that with a prudent choice of stepsizes $\{\alpha_t\}$, the gradient descent algorithm (GD) 
\begin{align}
	x_{t+1} = x_t - \frac{\alpha_t}{M} \nabla f(x_t) 
\end{align}
solves such a convex optimization problem to arbitrary accuracy from any initialization $x_0$. How quickly does GD converge? The mainstream approach (see e.g., the textbooks \citep{BertsekasNonlinear,Nocedal,nesterov-survey,vishnoi2021algorithms,bubeck-book,BNO,lan2020first,bertsekas2015convex} among many others) is to use a constant stepsize schedule $\alpha_t \equiv \bar{\alpha} \in (0,2)$ since this ensures
\begin{align}
	f(x_n) - f^* \leq \frac{cM\|x_0 - x^*\|^2}{n} 
	\label{eq:rate-std}
\end{align}
where $x^*$ denotes any minimizer of $f$, $f^* := f(x^*)$ denotes the corresponding minimal value, and $c$ is a small constant, e.g., $c = \tfrac{1}{4}$ for $\bar{\alpha} = 1$~\citep{DT14}.

\par The main question posed in Part I~\citep{alt23hedging1} was: can we accelerate the convergence of GD without changing the algorithm---just by judiciously choosing the stepsizes? Here we continue to investigate this question, now in the setting of smooth convex optimization. Note that this is markedly different from classical approaches to acceleration---starting from Nesterov's seminal result of 1983~\citep{nesterov-agd}, those approaches modify the basic GD algorithm by adding momentum, internal dynamics, or other additional building blocks beyond just changing the stepsizes. For this reason, we do not discuss that line of work in detail, and instead refer to~\citep{d2021acceleration}
 for a recent survey of this mainstream approach to acceleration, and to~\citep{alt23hedging1} for a full discussion of the relations between these approaches.

\subsection{Contribution}\label{ssec:cont}

\begin{figure}
	\centering
	\includegraphics[width=0.5\linewidth]{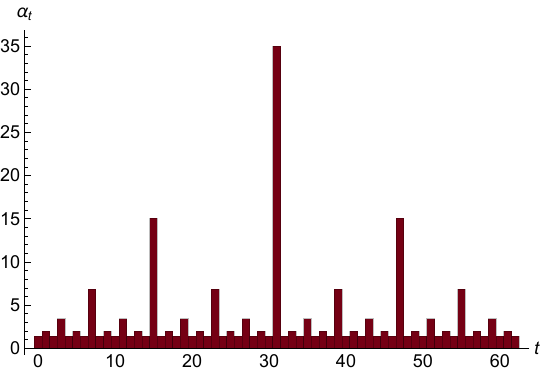}
	\caption{\footnotesize Silver Stepsize schedule $\{\alpha_0,\alpha_1,\alpha_2,\ldots\}$. See~\eqref{eq:steps} for the definition. Only the first $n=63$ values are shown (i.e., $k=8$). The fractal-like stepsizes are non-monotonic and have increasingly large ``spikes'' $\alpha_{2^k-1} = 1+\rho^{k-1}$.}
	\label{fig:silverstepsize}
\end{figure}

This paper provides a concise, self-contained proof that the Silver Stepsize Schedule proposed in Part~I~\citep{alt23hedging1} directly applies to smooth (non-strongly) convex optimization. This leads to faster convergence rates of GD for smooth convex optimization, as already pointed out in~\citep[\S1.1.4]{alt23hedging1}.

\par In this setting, the Silver Stepsize Schedule is particularly simple. For any integer $n = 2^{k}-1$, we recursively construct the schedule $h_{2n+1}$ of length $2n+1$ from the schedule $h_n$ of length $n$ via
\begin{align}
	h_{2n+1} := [h_n, \; 1 + \rho^{k-1}, \; h_n]
	\label{eq:steps-recursive}
\end{align}
where $\rho := 1 + \sqrt{2}$ denotes the silver ratio, and $h_1 := [\sqrt{2}]$. This results in the simple pattern $[\sqrt{2}, \; 2, \; \sqrt{2}, \; 1+\sqrt{2}, \dots]$ as depicted in Figure~\ref{fig:silverstepsize}. This schedule is exactly the Silver Stepsize Schedule from~\citep{alt23hedging1} in the limit that the strong convexity parameter vanishes (see Remark~\ref{rem:steps-limit} for details), and bears similarities to those in~\citep{gupta22,grimmer23,GrimmerShuWang}; see \S\ref{ssec:rel}.

We show that these stepsizes yield an improved convergence rate~\eqref{eq:rate-std} where $\tfrac{c}{n}$ is replaced by
\begin{align}
	r_k := \frac{1}{1 + \sqrt{4\rho^{2k} - 3}} \leq
	\frac{1}{2 \rho^{\log_2 n}}
	= \frac{1}{2n^{\log_2 \rho}}
	\approx \frac{1}{2 n^{1.2716}}
	\,.
	\label{eq:r-asymptotics}
\end{align}
This bound gives the correct asymptotic scaling of $r_k$ since the inequality is asymptotically tight.

\begin{theorem}[Main result]\label{thm:main}
	For any horizon $n = 2^k - 1$, any dimension $d$, any $M$-smooth convex function $f : \R^d \to \R$, and any initialization $x_0 \in \R^d$, 
	\begin{align}
		f(x_n) - f^* \leq r_k\, M\|x_0 - x^*\|^2 \,,
	\end{align}
	where $x^*$ denotes any minimizer of $f$, and $x_n$ denotes the output of $n$ steps of GD using the Silver Stepsize Schedule. In particular, in order to achieve error $f(x_n) - f^* \leq \eps$, it suffices to run GD for
	\begin{align*}
		n
		\leq
		\left( \frac{M \|x_0 - x^*\|^2}{2 \eps} \right)^{\log_{\rho} 2} 
		\approx
		\left( \frac{M\|x_0 - x^*\|2}{2 \eps} \right)^{0.7864}
		\;\; \text{iterations}.
	\end{align*}
\end{theorem}

We make several remarks. 1) This rate $n^{-\log_2 \rho} \approx n^{-1.2716}$ is intermediate between the textbook unaccelerated rate $n^{-1}$ and the accelerated rate $n^{-2}$ due to Nesterov in 1983~\citep{nesterov-agd}; see Figure~\ref{fig:threerates} for a visualization. 2) The Silver Stepsize Schedule is independent of the horizon, see \S\ref{sec:construction}. 3) We conjecture that, up to a constant factor, the rate $r_k$ is optimal among all possible stepsize schedules. This will be addressed in the forthcoming Part III.

\begin{figure}
	\centering
	\includegraphics[width=0.45\linewidth]{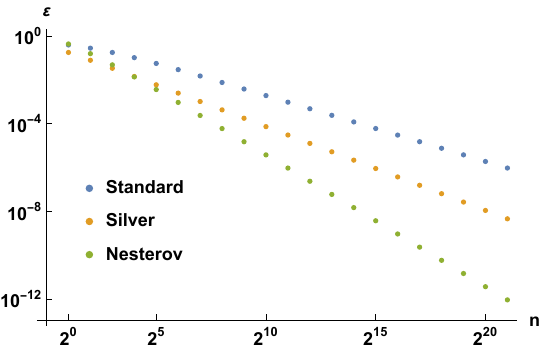}
	\caption{\footnotesize Upper bound on the optimality gap $f(x_n)-f^*$, as a function of the number of iterations $n$. The three plots correspond to the rates $O(n^{-1})$ for the standard constant stepsize $1/M$~\citep[Corollary 2.1.2]{nesterov-survey}, $O(n^{-\log_2 \rho}) \approx O({n^{-1.2716}})$ for the Silver Stepsize Schedule (Theorem~\ref{thm:main}), and $O(n^{-2})$ for Nesterov acceleration~\citep[Theorem 2.2.2]{nesterov-survey}.
 }
	\label{fig:threerates}
\end{figure}

\subsection{Related work }\label{ssec:rel}

Several time-varying stepsize schedules have been considered, e.g., Armijo-Goldstein rules, Polyak-type schedules, Barzilai-Borwein-type schedules, etc. However, until recently, no convergence analyses improved over the textbook unaccelerated rate except in the special case of minimizing convex quadratics. Here we discuss the recent line of work on accelerated GD via time-varying stepsizes. For brevity, we refer to Paper~I~\citep{alt23hedging1} for a more complete discussion of adjacent bodies of literature.

\par Beginning with Altschuler's 2018 MS thesis~\citep{altschuler2018greed}, a line of work designed time-varying stepsize schedules to achieve faster convergence rates for (non-quadratic) convex optimization. The thesis~\citep{altschuler2018greed} showed that time-varying stepsizes can lead to improvements over the textbook unaccelerated GD rate---by an optimal asymptotic factor in the separable setting by using random stepsizes, and by a constant factor in the strongly convex setting by giving optimal stepsizes for $n=2,3$. A key difficulty in the general non-separable setting is that the search for optimal stepsizes is non-convex and computationally difficult for larger horizons $n$.

In 2022,~\citet{gupta22} combined Branch \& Bound techniques with the PESTO SDP of~\citep{pesto} to develop algorithms that perform this search numerically, and as an example computed good approximate schedules in the convex setting for larger values of $n$ up to $50$. They observed a fit of roughly $O(n^{-1.178})$
and suggested from this that the asymptotic rate may be faster than the textbook rate $O(n^{-1})$. 

In July 2023,~\citet{grimmer23} showed how to prove asymptotic rates for the (non-strongly) convex setting by periodically cycling through finite schedules. While in the strongly convex setting composing progress from different cycles just amounts to multiplying contraction rates, in the non-strongly convex setting this can be more subtle depending on the approach.\footnote{Cf., our recursive gluing approach, which is a simple way of composing progress that unifies the convex (in \S\ref{sec:gluing}) and strongly convex settings (in Paper I). This is why our analysis is so compact and completely bypasses the machinery of straightforward patterns. See Remark~\ref{rem:formulation}.}
To deal with this, he introduced the notion of straightforward stepsize patterns and obtained a constant-factor improvement over the textbook unaccelerated rate by cycling through approximate schedules of length $n=127$. From these numerics, he conjectured that further optimizing stepsizes might lead to an asymptotic rate of $O((n \log n)^{-1})$, a milder improvement than that conjectured by~\citep{gupta22}. 

\par In September 2023, two concurrent papers appeared~\citep{GrimmerShuWang,alt23hedging1}.~\citet{alt23hedging1} was Part~I: there we proposed the Silver Stepsize Schedule of arbitrary size to prove asymptotic acceleration for the strongly convex setting. This improved the textbook unaccelerated rate $\Theta(\kappa)$ to $\Theta(\kappa^{\log_2 \rho}) \approx \Theta(\kappa^{0.7864})$ where $\kappa$ is the condition number. We conjectured and provided partial evidence that these rates are optimal among all possible stepsize schedules. This result was achieved by introducing the technique of recursive gluing to establish multi-step descent, which we make use of here. The other paper was~\citet{GrimmerShuWang}. By utilizing a certain non-periodic sequence of increasingly large stepsizes and building upon the ``straightforwardness'' machinery in~\citep{grimmer23}, they proved the rate $O(n^{-1.0245})$ for the convex setting in v1, later improved to $O(n^{-1.0564})$ in v2. 

\par The results in Part I~\citep{alt23hedging1} were stated for the strongly convex setting. As pointed out throughout that paper, this is not a restriction because on the one hand those results immediately imply analogously accelerated rates 
$\Theta(n^{-\log_2 \rho}) \approx \Theta(n^{-1.2716})$
for the convex setting via a standard black-box reduction; and on the other hand, this reduction can be bypassed by re-doing the analysis for the convex setting since all the core conceptual ideas extend directly~\citep[\S1.1.4]{alt23hedging1}.
The present paper provides these details.

\subsection{Notation}\label{ssec:notation}

\paragraph*{Indexing.} Throughout, the horizon is $n = 2^k-1$. This correspondence between $n \in \{1,3,7,15,\dots\}$ and $k \in \{1,2,3,4,\dots\}$ allows re-indexing in a way that simplifies notation for the recursion.

\paragraph*{Rescaling.} For simplicity, we normalize $M = 1$, i.e., $\| \nabla f(x) - \nabla f(y) \| \leq \|x-y\|$ for all $x,y$. This is without loss of generality since if $h$ is $M$-smooth and convex,  then $f := h/M$ is $1$-smooth and convex, thus the result we establish $f_n - f^* \leq r_k\|x_0 - x^*\|^2$ implies $h_n - h^* \leq r_k M \|x_0 - x^*\|^2$. Note that running GD on $f$ simply amounts to rescaling the Silver Stepsize Schedules for $h$ by $1/M$.
	  
\section{Silver Stepsize Schedule}\label{sec:construction}

 The Silver Stepsize Schedule is defined recursively in~\eqref{eq:steps-recursive}. Here we mention an equivalent direct expression and make several remarks. Let $\bin(t)$ denote the $2$-adic valuation of $t$, i.e., the smallest non-negative integer $i$ such that $2^i$ is in the binary expansion of $t$. For example, $\bin(1) = 0$, $\bin(2) = 1$, $\bin(3) = 0$, $\bin(4) = 2$, etc.

\begin{defin}[Silver Stepsize Schedule for smooth convex optimization]
	For $t \in \{0,1, 2,\dots\}$, the $t$-th stepsize of the Silver Stepsize Schedule is
		\begin{align}
		\alpha_t
		:=
		1 + \rho^{\bin(t+1)-1}\,.
		\label{eq:steps}
	\end{align}
\end{defin}
This schedule is non-monotonic, fractal-like, and has increasingly large spikes that grow exponentially (by a factor of $\rho$) yet become exponentially less frequent (by a factor of $2$). See Figure~\ref{fig:silverstepsize}. It can also be easily implemented\footnote{For instance, in Python the command \texttt{[1+rho**((k \& -k).bit\_length()-2) for k in range(1,64)]} generates the first $63$ steps of the Silver Schedule shown in Figure~\ref{fig:silverstepsize}.} in any computer language.

\begin{remark}[Limit of Silver Stepsize Schedules in the strongly convex case]\label{rem:steps-limit}
	The schedule~\eqref{eq:steps} is simply the Silver Stepsize Schedule for smooth \emph{strongly}-convex optimization in Part~I~\citep{alt23hedging1}, in the limit as the strong convexity parameter tends to~$0$.
	The stepsizes simplify in the limit: they increase by a factor of $\rho$, shorter schedules are prefixes of longer schedules, neither $a_1$ nor any of the $b_n$ sequence in~\citep{alt23hedging1} is needed, and they are non-periodic (hence why there is no rate saturation here).
\end{remark}

We also record a simple closed-form expression for the sum of the first $n=2^k -1$ Silver Stepsizes. 

\begin{lemma}[Sum of Silver Stepsizes]\label{lem:sum-steps}
	$\sum_{t=0}^{n-1} \alpha_t = \rho^k - 1$ for any $k \in \N$. 
\end{lemma}
\begin{proof}
	The base case $k=1$ is trivial. The inductive step follows from the recursion~\eqref{eq:steps-recursive}.
\end{proof}

\section{Recursive gluing}\label{sec:gluing}

Here we prove that the Silver Stepsize Schedule has convergence rate $r_k$ for smooth convex optimization. The analysis closely mirrors the strongly convex setting in Part~I~\citep{alt23hedging1}: we prove the advantage of time-varying stepsizes via multi-step descent rather than iterating the greedy 1-step bound, show multi-step descent by exploiting long-range consistency conditions along the GD trajectory, certify multi-step descent via recursive gluing, and recursively glue by combining the same three components in the same way. In the interest of brevity, we refer to~\citep{alt23hedging1} for a detailed discussion of all these concepts. 

\par Briefly, the idea behind multi-step descent is that it is essential to capture how different iterations affect other iterations' progress. We do this by exploiting long-range consistency conditions between the iterates along GD's trajectory, as encoded by the co-coercivities 
\begin{align}
	Q_{ij} := 2(f_i - f_j) + 2\langle g_j, x_j - x_i \rangle - \|g_j - g_i\|^2\,.
	\label{eq:cocoercivity}
\end{align}
The significance of these co-coercivities is that the constraints $\{Q_{ij} \geq 0\}_{i \neq j \in \{0,1,\dots,n,*\}}$ are necessary and sufficient for the existence of a $1$-smooth convex function $f$ satisfying $f_i = f(x_i)$ and $g_i = \nabla f(x_i)$ for each $i \in \{0,1,\dots,n,*\}$~\citep{pesto}. In other words, the co-coercivity conditions $\{Q_{ij} \geq 0 \}_{i \neq j \in \{0,1,\dots,n,*\}}$ generate all possible long-range consistency constraints on the objective function $f$. Or, said another way, the co-coercivity conditions generate all possible valid inequalities with which one can prove convergence rates for GD. For a further discussion, see~\citep[\S2.2]{alt23hedging1}.

\par Concretely, to prove Theorem~\ref{thm:main}, we exhibit explicit non-negative multipliers $\lam_{ij}$ satisfying
\begin{align}
	\sum_{i \neq j \in \{0, \dots, n,*\}} \lam_{ij} Q_{ij}  = 
	\|x_0 - x^*\|^2 - \|x_n - c_k g_n - x^*\|^2
	+ \frac{f^* - f_n}{r_k}\,,
	\label{eq:cert}
\end{align}
where we use the shorthand $c_k := \frac{1}{2r_k}$. Since $\sum_{ij}\lam_{ij} Q_{ij} \geq 0$ for any $1$-smooth convex function, and since $\|x_n - c_k g_n - x^*\|^2 \geq 0$ trivially as it is a square, this immediately implies the desired rate 
\begin{align}
	f_n - f^* \leq r_k \|x_0 - x^*\|^2\,.
	\label{eq:cert-cor}
\end{align}

\begin{remark}[The importance of a composable formulation]\label{rem:formulation}
We emphasize that while there are several alternative formulations to~\eqref{eq:cert} that imply a rate like~\eqref{eq:cert-cor} after dropping terms, our formulation~\eqref{eq:cert} is ``canonical'' because it composes well under recurrence (Theorem~\ref{thm:cert}). This enables a quite short and simple proof. Another reason this formulation is particularly nice is because the base case $n=0$ amounts to a re-writing of the definition of $Q_{*0}$, which is an improvement over the standard bound $f_0 - f^* \leq \half \|x_0 - x^*\|^2$. In fact, this improvement is precisely what makes  recursive gluing work so seamlessly.
\end{remark}

\begin{example}[$n=0$]\label{ex:n=0}
	For $n=0$, the identity~\eqref{eq:cert} 
	is
	\begin{align*}
		Q_{*0} = \|x_0 - x^*\|^2 - \|x_0 - g_0 - x^*\|^2 + 2(f^* - f_0)\,.
	\end{align*}
	 Here, $r_0 = \half$, $c_0 = 1$, and the only non-zero multiplier is $\lam_{*0} = 1$. 
\end{example}

\begin{example}[$n=1$]\label{ex:n=1}
	For $n=1$, the identity~\eqref{eq:cert} is
	\begin{align*}
		\sum_{i \neq j \in \{0, 1, *\}} \lam_{ij} Q_{ij} 
		=
		\|x_0 - x^*\|^2 - \|x_1 - c_1 g_1 - x^*\|^2 + \frac{f^* - f_1}{r_1}\,.
	\end{align*}
	Here, $x_1 = x_0 - \sqrt{2} g_0$, $r_1 
	= \frac{1}{1+\sqrt{4\rho^2-3}}
	\approx 0.1816$, $c_1 = \frac{1}{2r_1} \approx 2.7535$, and 
	\begin{align*}
		\begin{bmatrix}
			\lam_{00} & \lam_{01} & \lam_{0*} \\
			\lam_{10} & \lam_{11} & \lam_{1*} \\
			\lam_{*0} & \lam_{*1} & \lam_{**}
		\end{bmatrix}
		=
		\begin{bmatrix}
			0
			&
			1+\sqrt{2}
			&
			0
			\\
			1
			&
			0
			&
			\sqrt{2}
			\\
			\sqrt{2}
			&
			\frac{1}{2r_1}
			&
			0
		\end{bmatrix}
		\,.
	\end{align*}
\end{example}

\par For larger horizons $n$, we construct $\lam_{ij}$ via the recursive gluing technique of~\citep{alt23hedging1}. See Figure~\ref{fig:gluing}. Below, we say that the multipliers $\{\sig_{ij}\}_{i,j \in \{0,\dots,n,*\}}$ satisfy the $*$-sparsity property if $\sig_{i*} = 0$ for all $i < n$. This is satisfied by construction and simplifies part of the proof (isolated in Lemma~\ref{lem:opt}).

\begin{figure}
	\centering
	\includegraphics[width=0.4\linewidth]{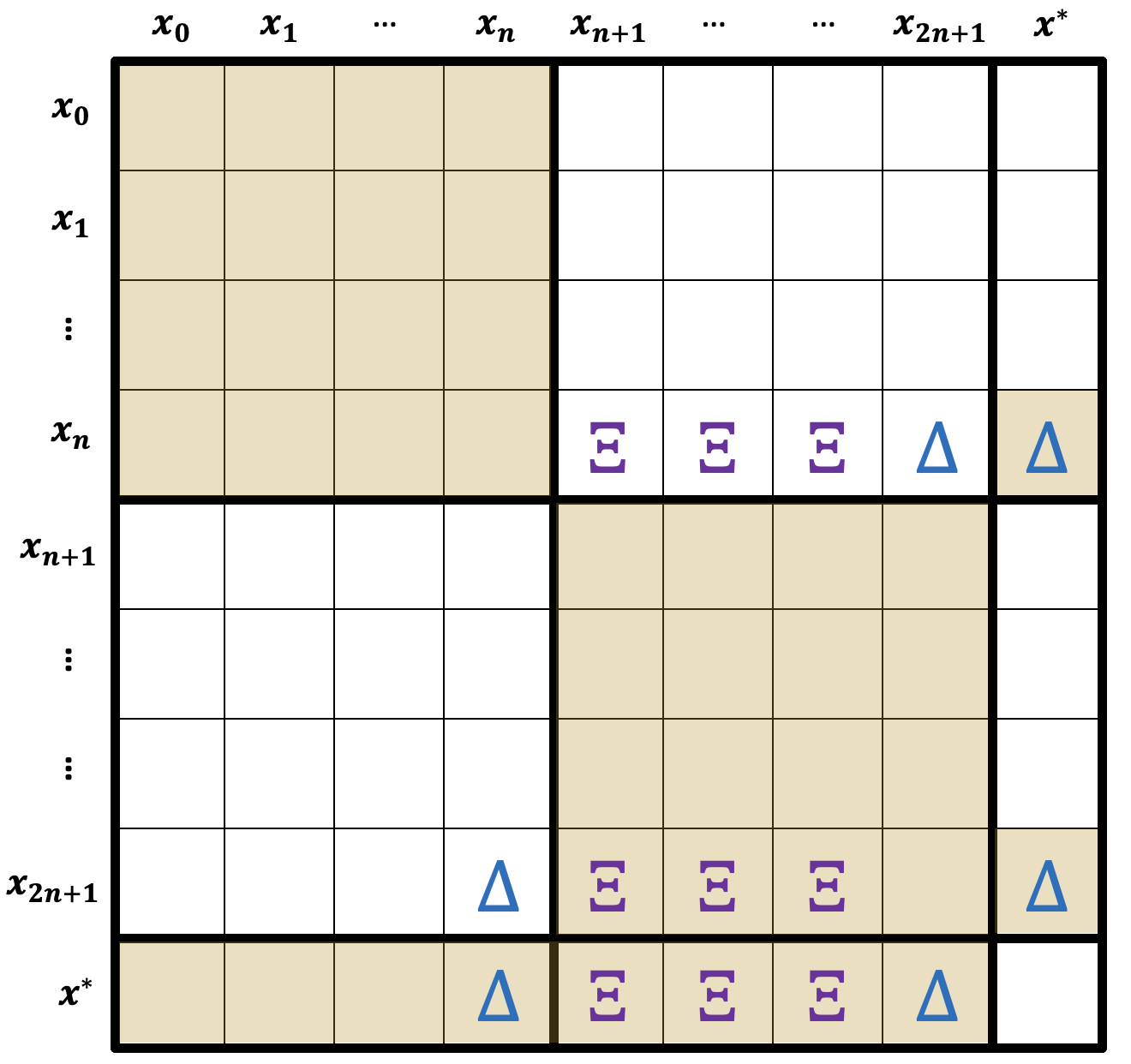}
	\caption{\footnotesize Components of the recursively glued certificate in Theorem~\ref{thm:cert}, illustrated here for combining two copies of the $n=3$ certificate (shaded) to create the $2n+1 = 7$ certificate. The structure of this recursive gluing is identical to the strongly convex setting from Part I~\citep{alt23hedging1}, modulo re-indexing for horizons of the form $n = 2^k -1 $ rather than $2^k$.}
	\label{fig:gluing}
\end{figure}

  \begin{theorem}[Recursive gluing]\label{thm:cert}
	Let $n = 2^k-1$. Suppose $\{\sig_{ij}\}_{i,j \in \{0,\dots,n,*\}} $ satisfies $*$-sparsity and certifies the $n$-step rate, i.e.,
	\begin{align}
		\sum_{i \neq j \in \{0, \dots, n,*\}} \sig_{ij} Q_{ij} 
		=
		\|x_0 - x^*\|^2 - \|x_n - c_k g_n - x^*\|^2  + \frac{f^* - f_n}{r_k}\ .
		\label{eq:cert:n}
	\end{align}
	Then there exists $\{\lam_{ij}\}_{i, j \in \{0,\dots,2n+1,*\}}$ that satisfies $*$-sparsity and certifies the $2n+1$-step rate, i.e.,
	\begin{align}
		\sum_{i \neq j \in \{0, \dots, 2n+1,*\}} \lam_{ij} Q_{ij} 
		=
		\|x_0 - x^*\|^2 - \|x_{2n+1} - c_{k+1} g_{2n+1} - x^*\|^2 + \frac{f^* - f_{2n+1}}{r_{k+1}}\ .
		\label{eq:cert:2n+1}
	\end{align}
	Moreover, this certificate is explicitly given by
	\begin{align}
		\lam_{ij} = \underbrace{\Theta_{ij}}_{\text{gluing}}
		+
		\underbrace{\Xi_{ij}}_{\text{rank-one correction}}
		+
		\underbrace{\Delta_{ij}}_{\text{sparse correction}}
		\label{eq:lambda}
	\end{align}
	The ``gluing component'' $\Theta$ is defined as 
	\begin{align}
		\Theta_{ij} := \underbrace{\sig_{i,j} \cdot \mathds{1}_{i,j \in \{0,\dots,n,*\}}}_{\text{recurrence for first $n$ steps}}
		\;\;+\;\;
		\underbrace{ (1 + 2\rho)\; \sig_{i-n-1,j-n-1} \cdot \mathds{1}_{i,j \in \{n+1,\dots,2n+1,*\}}}_{\text{recurrence for last $n$ steps}}\,.
		\label{eq:Theta}
	\end{align}
	The ``rank-one correction'' $\Xi$ is zero except the entries $\{\Xi_{ij}\}_{i \in \{n,2n+1,*\},\; j \in \{n+1, \dots, 2n\}}$ which are
	\begin{align}
		\begin{bmatrix}
			\Xi_{n,n+1} & \Xi_{n,n+2} & \cdots & \Xi_{n,2n} \\
			\Xi_{2n+1,n+1} & \Xi_{2n+1,n+2} & \cdots & \Xi_{2n+1,2n} \\
			\Xi_{*,n+1} & \Xi_{*,n+2} & \cdots & \Xi_{*,2n}
		\end{bmatrix}
		:=
		\rho \begin{bmatrix}
			1 \\
			1 \\
			-2
		\end{bmatrix} \begin{bmatrix}
			\alpha_{n+1} & \alpha_{n+2} & \cdots & \alpha_{2n}
		\end{bmatrix}
		\label{eq:Xi}
	\end{align}
	The ``sparse correction'' $\Delta$ is zero except the entries $\{\Delta_{ij}\}_{i \neq j \in \{n,2n+1,*\}}$ which are
	\begin{align}
		\begin{bmatrix}
			\Delta_{n,n}
			&
			\Delta_{n,2n+1}
			&
			\Delta_{n,*} 
			\\
			\Delta_{2n+1,n}
			&
			\Delta_{2n+1,2n+1} 
			&
			\Delta_{2n+1,*} 
			\\
			\Delta_{*,n}
			&
			\Delta_{*,2n+1} 
			&
			\Delta_{*,*}
		\end{bmatrix}
		:=
		\begin{bmatrix}
			0 
			&
			\rho
			&
			1 - \rho^k
			\\
			\rho^k
			&
			0 
			&
			2\rho - \sqrt{2} \rho^{k+1}
			\\
			1 + \rho^{k-1} - \frac{1}{2r_k}
			&
			\frac{1}{2r_{k+1}} - \frac{1+2\rho}{2r_k}
			&
			0
		\end{bmatrix}
		\label{eq:Del}
	\end{align}
\end{theorem}

We remark that using this recursion, one can also write out the $n$-step certificate directly. The value of each multiplier then depends on the binary expansion of its indices.

This recursive gluing immediately implies the main result of the paper, Theorem~\ref{thm:main}. 

\begin{proof}[Proof of Theorem~\ref{thm:main}]
	It suffices to prove~\eqref{eq:cert}; we prove this by induction. The base case $n=1$ is Example~\ref{ex:n=1}. The inductive step is Theorem~\ref{thm:cert}. 
\end{proof}

The rest of the section is dedicated to proving Theorem~\ref{thm:cert}. We first isolate two helper lemmas in \S\ref{ssec:gluing:helper}, and then we combine them to prove the result in \S\ref{ssec:gluing:pf}.  For notational simplicity, henceforth we assume $x^* = 0$; this is without loss of generality after translating. 

\subsection{Helper lemmas}\label{ssec:gluing:helper}

Here we provide three helper lemmas for the proof of Theorem~\ref{thm:cert}. The first lemma explicitly computes all multipliers involving $x^*$ for the $n$-step certificate. 

\begin{lemma}[Multipliers involving $x^*$]\label{lem:opt} 
	$\sig_{n*} = \rho^k - 1$, $\sig_{*n} = \frac{1}{2r_k}$, and $\sig_{*t} = \alpha_t$ for $t \in \{0,\dots,n-1\}$. 
\end{lemma}
\begin{proof}
	Expand both sides of the identity~\eqref{eq:cert:n} using the $*$-sparsity assumption, the definition of the co-coercivities, and the definition of GD, i.e., $x_t = x_0 - \sum_{s=0}^{t-1} \alpha_s g_s$. Matching coefficients for the $\langle x_0, g_t \rangle$ terms yields the claimed formulas for $\sig_{*t}$. Matching coefficients for the $f^*$ term gives the identity $\sig_{n*} = \sum_{j=0}^n \sig_{*j}-  \tfrac{1}{2r_k}$. This equals $\rho^k - 1$ by Lemma~\ref{lem:sum-steps}. 
\end{proof}

The next two lemmas are more substantial. These help us verify~\eqref{eq:cert:2n+1}---which consists of a linear form in all $2n+3$ function values and a quadratic form in all $2n+3$ gradients and iterates. Na\"ively verifying such an identity requires checking $\Theta(n)$ coefficients for the linear form and $\Theta(n^2)$ coefficients for the quadratic form. The following two lemmas show that due to the recursive construction of the Silver Stepsize Schedule and the gluing, these forms only affect the indices $n,2n+1,*$. This reduces verifying the rate certificate~\eqref{eq:cert:2n+1} to checking only $\Theta(1)$ coefficients, as detailed below in \S\ref{ssec:gluing:pf}. Below, for shorthand, let $F_{ij} := 2(f_i - f_j)$ and $P_{ij} := 2\langle g_j, x_j - x_i \rangle - \|g_i - g_j\|^2$  denote the linear and quadratic components of $Q_{ij}$, respectively. For bookkeeping purposes, we use $3$-dimensional vectors and $4 \times 4$ matrices to denote the coefficients of these forms.

\begin{lemma}[Succinct linear forms]\label{lem:succinct-linear}
	Let $u := [f_n, f_{2n+1}, f^*]^T$, and let $e, s, \ell \in \R^3$ be the vectors defined in Appendix~\ref{app:succinct-linear}. 
	\begin{itemize}
		\item \underline{Gluing error.} $\frac{f^* - f_{2n+1}}{r_{k+1}} - \sum_{ij} \Theta_{ij} F_{ij} = \langle e, u \rangle$
		\item \underline{Sparse correction.} $\sum_{ij} \Delta_{ij} F_{ij} = \langle s, u \rangle$
		\item \underline{Rank-one correction.} $\sum_{ij} \Xi_{ij} F_{ij} = \langle \ell, u \rangle$
	\end{itemize}
\end{lemma}
\begin{proof}
	Expand the definition of co-coercivities, simplify the rank-one correction using Lemma~\ref{lem:sum-steps}, and simplify the sparse correction using the Pell recurrence $\rho^{k+1} = 2\rho^k + \rho^{k-1}$.
\end{proof}

\begin{lemma}[Succinct quadratic forms]\label{lem:succinct} Let $v := [x_n,g_n,x_{2n+1},g_{2n+1}]^T$, and let $E$, $S$, $L$ be the $4 \times 4$ matrices defined in Appendix~\ref{app:succinct}.
	\begin{itemize}
		\item \underline{Gluing error.} 
		$\|x_0\|^2 - \|x_{2n+1} - c_{k+1} g_{2n+1}\|^2
		- \sum_{ij} \Theta_{ij} P_{ij} = \langle E, vv^T \rangle$
		\item \underline{Sparse correction.} $\sum_{ij} \Delta_{ij} P_{ij} = \langle S, vv^T \rangle$
		\item \underline{Rank-one correction.} $\sum_{ij} \Xi_{ij} P_{ij} = \langle L, vv^T \rangle$
	\end{itemize}
\end{lemma}
\begin{proof}
	Deferred to Appendix~\ref{app:succinct} for brevity.
\end{proof}

\subsection{Proof of recursive gluing (Theorem~\ref{thm:cert})}\label{ssec:gluing:pf}

\paragraph*{Non-negativity.} We verify $\lam_{ij} \geq 0$ for all $i \neq j \in \{0, \dots, 2n+1,*\}$. For nearly all entries, this is obvious since $\lam_{ij}$ is constructed by adding and multiplying non-negative numbers. For the remaining entries where either $\Xi$ or $\Delta$ is negative, compute the corresponding entry of $\lam$ by summing the corrections and using Lemma~\ref{lem:opt}. This gives $\lam_{n*} = 0$, $\lam_{2n+1,*} = \rho^{k+1} - 1$, $\lam_{*,n} = \rho^{k-1} + 1$, $\lam_{*,2n+1} = \tfrac{1}{2r_{k+1}}$, and $\lam_{*,t} = \alpha_t$ for all $t \in \{n+1,\dots,2n\}$. All these entries are clearly non-negative.

\paragraph*{Rate certificate.} The identity~\eqref{eq:cert:2n+1} has two components: a linear form in the function values and a quadratic form in the iterates and gradients. For the linear form, it suffices to verify $e-s-\ell = 0$ by Lemma~\ref{lem:succinct-linear}, where $e,s,\ell \in \R^3$ are the vectors defined in Appendix~\ref{app:succinct-linear}. This is obvious by inspection. For the quadratic form, it suffices to verify $E - S - L = 0$, where $E,S,L \in \R^{4 \times 4}$ are the matrices in Lemma~\ref{lem:succinct} defined in Appendix~\ref{app:succinct}. By plugging in the explicit values for $r_k$ and $\Delta$, this is straightforward to check by hand. For brevity, we provide a simple Mathematica script that verifies these identities at the URL~\citep{MathematicaURL}.

	 \appendix
	 
\section{Deferred proof details}\label{app:deferred}

\subsection{Succinct linear forms (Lemma~\ref{lem:succinct-linear})}\label{app:succinct-linear}

The vectors $e, s, \ell$ in Lemma~\ref{lem:succinct-linear} are defined as follows:
\begin{align*}
	\small
	e := \begin{bmatrix} 
		\frac{1}{r_k} \\
		-\frac{1}{r_{k+1}} + \frac{1+2\rho}{r_k} \\
		\frac{1}{r_{k+1}} - \frac{4\rho}{r_k}
	\end{bmatrix}
	\,,
	\qquad
	s := 2\rho(\rho^{k} - 1) \begin{bmatrix}
		1 \\
		1 \\
		-2
	\end{bmatrix}
	\,,
	\qquad
	\ell := \begin{bmatrix}
		\frac{1}{r_k} - 2\rho(\rho^k - 1)\\
		-\frac{1}{r_{k+1}} + \frac{1+2\rho}{r_k} - 2\rho(\rho^k - 1) \\
		\frac{1}{r_{k+1}} - \frac{4\rho}{r_k} + 4\rho(\rho^k - 1)
	\end{bmatrix}\,.
\end{align*}

\subsection{Succinct quadratic forms (Lemma~\ref{lem:succinct})}\label{app:succinct}

The matrices in Lemma~\ref{lem:succinct} are defined as follows. For brevity, we write $\sim$ for the entries below the diagonal since these coefficient matrices are symmetric.
\begin{align*}
	\small
	E :=
	\begin{bmatrix}
		-2\rho
		&
		(1 + 2\rho) (1 + \rho^{k-1}) -\frac{1}{2r_k} 
		&
		0
		&
		0
		\\
		\sim
		&
		\frac{1}{4r_k^2} - (1 + 2\rho) (1+\rho^{k-1})^2
		&
		0
		&
		0 
		\\
		\sim
		&
		\sim
		&
		2\rho 
		&
		\frac{1}{2r_{k+1}} - \frac{1+2\rho}{2r_k}
		\\
		\sim & \sim & \sim &
		\frac{1+2\rho}{4r_k^2} - \frac{1}{4r_{k+1}^2}
	\end{bmatrix}
\end{align*}
\begin{align*}
	\small
	S := 
	\begin{bmatrix}
		0
		&
		\Delta_{2n+1,n} + \Delta_{*,n}
		&
		0
		&
		-\Delta_{n,2n+1}
		\\
		\sim
		&
		-\left( \Delta_{n,*} + \Delta_{*,n} + \Delta_{2n+1,n} + \Delta_{n,2n+1} \right)
		&
		-\Delta_{2n+1,n}
		&
		\Delta_{n,2n+1} + \Delta_{2n+1,n}
		\\
		\sim
		&
		\sim
		&
		0
		&
		\Delta_{n,2n+1} + \Delta_{*,2n+1}
		\\
		\sim
		&
		\sim
		&
		\sim
		&
		- \left( \Delta_{n,2n+1} + \Delta_{2n+1,n}  + \Delta_{2n+1,*} + \Delta_{*,2n+1}\right)
	\end{bmatrix}
\end{align*}
\begin{align*}
	\small
	L :=  \rho \begin{bmatrix}
		-2
		&
		2 + \rho^{k-1} 
		&
		0
		&
		1
		\\
		\sim
		&
		-1 - \rho^k - 2 \rho^{k-1} 
		&
		\rho^{k-1}
		&
		- 1 - \rho^{k-1}
		\\
		\sim
		&
		\sim
		&
		2
		&
		-1
		\\
		\sim
		&
		\sim
		&
		\sim
		&
		1 - \rho^k
	\end{bmatrix}
\end{align*}

\begin{proof}[Proof of Lemma~\ref{lem:succinct}]
	\underline{Gluing component.} 
	Recall that $h_{2n+1} = [h_n, \alpha_n, h_n]$ by the recursive construction~\eqref{eq:steps-recursive} of the Silver Stepsize Schedule. Thus, since $\sig$ certifies the $n$-step rate, we have
	\begin{align*}
		\sum_{i \neq j \in \{0, \dots, n,*\}} \sig_{ij} P_{ij} 
		&=
		\|x_0\|^2 - \|x_n - c_k g_n\|^2
		\\ 
		\sum_{i \neq j \in \{n+1, \dots, 2n,*\}} \sig_{ij} P_{ij} 
		&=
		\|x_{n+1}\|^2 - \|x_{2n+1} - c_k g_{2n+1}\|^2
	\end{align*}
	By definition, $\sum_{ij} \Theta_{ij} P_{ij}$ is the former plus $(1 + 2\rho)$ times the latter. 
	Thus the gluing error $\|x_0\|^2 - \|x_{2n+1} - c_{k+1} g_{2n+1}\|^2 - \sum_{ij} \Theta_{ij} P_{ij}$ is equal to
	\begin{align*}
		(1+2\rho) \|x_{2n+1} - c_k g_{2n+1}\|^2 - \|x_{2n+1} - c_{k+1} g_{2n+1}\|^2 + \|x_n - c_k g_n\|^2 - (1 + 2\rho) \|x_{n+1}\|^2
	\end{align*}
	The key point is that after 
	expanding $x_{n+1} = x_n - \alpha_n g_n$, this is a quadratic form in only the $4$ variables $x_n, g_n, x_{2n+1},g_{2n+1}$. Collecting terms, substituting $c_k = \tfrac{1}{2r_k}$, and using the definition of the Silver Stepsize $\alpha_n = 1 + \rho^{k-1}$ yields $\langle E, vv^T \rangle$, as desired.
	
	\par \underline{Sparse component.} Expand the definition of the co-coercivity and collect terms. Note also that this equals the matrix $S$ in Part~I~\citep{alt23hedging1}, in the limit that the strong convexity tends to $0$.
	
	\par \underline{Low-rank component.}  Expanding the co-coercivities, using the identity $x_{2n+1} - x_{n+1}= - \sum_{t=n+1}^{2n} \alpha_t g_t$ by definition of GD, and using the identity $\sum_{t=n+1}^{2n} \alpha_t = \sum_{t=0}^{n-1} \alpha_t = \rho^k - 1$ from Lemma~\ref{lem:sum-steps}, 
	\begin{align*}
		\sum_{t=n+1}^{2n} \Big( \Xi_{n,t} P_{n,t} - \Xi_{*,t} P_{*,t} \Big)
		=
		\rho \sum_{t=n+1}^{2n} \alpha_t \Big(
		2 \langle g_t, g_n - x_n \rangle - \|g_n\|^2
		\Big)
		=
		2\rho \langle x_{2n+1} - x_{n+1}, x_n - g_n \rangle - \rho(\rho^k - 1 ) \|g_n\|^2 \,.
	\end{align*}
	An identical calculation yields
	\begin{align*}
		\sum_{t=n+1}^{2n} \Big( \Xi_{2n+1,t} P_{2n+1,t} - \Xi_{*,t} P_{*,t} \Big)
		&=
		2\rho \langle x_{2n+1} - x_{n+1}, x_{2n+1} - g_{2n+1} \rangle - \rho(\rho^k - 1 ) \|g_{2n+1}\|^2 \,.
	\end{align*}
	Now sum these two displays, use the definition of GD to expand $x_{n+1} = x_n - \alpha_n g_n$, and use the definition of the Silver Stepsize $\alpha_n = 1 + \rho^{k-1}$. This gives $\langle L, vv^T \rangle$, as desired.
\end{proof}

	 \footnotesize
	 \addcontentsline{toc}{section}{References}
	 \bibliographystyle{plainnat}
	 \bibliography{hedging}{}

\end{document}